\documentclass[english]{ourlematema}
\usepackage{amsmath, amsthm, amssymb, hyperref, color}
\usepackage{MnSymbol} 
\usepackage{graphicx}
\usepackage{caption}
\usepackage[all,cmtip]{xy}
\usepackage{verbatim}
\usepackage{chemfig, chemnum}
\usepackage{tikz}
\usepackage[linesnumbered,lined,commentsnumbered]{algorithm2e}
\usepackage{caption}
\usepackage{subcaption}
\usepackage{enumitem}

\usepackage{makecell}
\usepackage{array}

\newtheorem{theorem}{Theorem}
\numberwithin{theorem}{section}
\newtheorem{proposition}[theorem]{Proposition}
\newtheorem{lemm}[theorem]{Lemma}
\newtheorem{corollary}[theorem]{Corollary}

\newtheorem{remark}[theorem]{Remark}
\newtheorem{example}[theorem]{Example}

\theoremstyle{definition}

\newcommand{\PP}{\mathbb{P}}
\newcommand{\RR}{\mathbb{R}}

\newcommand{\CC}{\mathbb{C}}

\newcommand{\SSS}{\mathbb{S}}

\newcommand{\Ls}{\mathcal{L}}
\newcommand{\Linv}{\mathcal{L}^{-1}}

\DeclareMathOperator{\rad}{rad}
\DeclareMathOperator{\spa}{span}
\DeclareMathOperator{\rk}{rk}
\DeclareMathOperator{\codim}{codim}
\DeclareMathOperator{\adj}{adj}

\DeclareMathOperator{\GL}{GL}
\DeclareMathOperator{\Diag}{Diag}

\DeclareMathOperator{\Gr}{Gr}
\DeclareMathOperator{\Jo}{Jo}

 \title{Jordan Algebras of Symmetric Matrices}

  \author{Arthur Bik}
  \address{%
MPI for Mathematics in the Sciences, Leipzig \\
\email{arthur.bik@mis.mpg.de}
}
  \author{Henrik Eisenmann}
  \address{%
MPI for Mathematics in the Sciences, Leipzig \\
\email{henrik.eisenmann@mis.mpg.de}
}

\author{ Bernd Sturmfels}
\address{%
  MPI for Mathematics in the Sciences, Leipzig \\
\email{bernd@mis.mpg.edu }
}

\date{2021/05/17}

\begin{document}

\begin{abstract}
\noindent
We study linear spaces of symmetric matrices whose reciprocal is also a linear space.
These are Jordan algebras. We classify such
algebras in low dimensions, and 
we study the associated Jordan loci in the Grassmannian.
\end{abstract}
\maketitle

\section{Introduction}

Let $\SSS^n$ be the space of symmetric $n \times n$ matrices over the complex numbers $\CC$.
The $m$-dimensional subspaces $\mathcal{L} \subset \SSS^n$
are points in the Grassmannian ${\rm Gr}(m,\SSS^n)$. The terms
{\em pencil of quadrics} \cite{FMS}  resp.~{\em net of quadrics}~\cite{Wall77} are used when $m=2,3$.
We here consider {\em regular}  $\mathcal{L}$, i.e. 
$\mathcal{L}_{\rm inv} = \{ X \in \mathcal{L} \mid {\rm det}(X) \not = 0 \}$ is nonempty.
The {\em reciprocal variety} $\mathcal{L}^{-1}$ is the closure in $\SSS^n$ of the set 
$\{X^{-1} \mid X \in \mathcal{L}_{\rm inv}\}$.
We wish to know under which circumstances
the variety $\mathcal{L}^{-1}$ is a linear space in $\SSS^n$.
The motivation for this question arises in optimization,
namely in the theory of semidefinite programming \cite{Fay, pablo},
and in statistics, namely from  Gaussian models
that are linear in covariance matrices and  concentration matrices \cite{Jen, STZ}.

\smallskip

The following theorem furnishes a complete answer to our question.\vspace{-0.2cm}

\begin{theorem} \label{thm:characterization}
For $\mathcal{L} \in {\rm Gr}(m,\SSS^n)$ and $U\in\Ls_{\rm inv}$,
 the following are equivalent: \vspace{-0.17cm}
 \begin{itemize}
 \item[(a)] The reciprocal variety $\mathcal{L}^{-1}$ is
 also a linear space in $\SSS^n$. \vspace{-0.17cm}
 \item[(b)] $\mathcal{L}$ is a subalgebra 
of the Jordan algebra $(\SSS^n,\bullet_U)$.\vspace{-0.17cm}
  \item[(c)] $\mathcal{L}^{-1}$ equals $\mathcal{L}$ up to congruence; namely
  $\Linv=U^{-1}\Ls U^{-1}$.
 \end{itemize}
\end{theorem}\vspace{-0.05cm}

Theorem \ref{thm:characterization} is due to Jensen \cite{Jen}
in the special case when $\mathcal{L}$ contains the
 identity matrix ${\bf 1}_n$. He then uses
$U = {\bf 1}_n$ for the Jordan algebra structure on $\SSS^n$.\vspace{-0.2cm}

\begin{example}[$m=n=4$] 
Consider the two linear subspaces of $\SSS^4$ given by \vspace{-0.2cm}
$$  \mathcal{L}_1 \,\,=\,\,
\begin{small} \begin{pmatrix}  x & y & 0 & 0 \\
 y & z & 0 & 0 \\
 0 &  0 & w & 0 \\
 0 & 0 & 0 & w \end{pmatrix}\end{small} \quad {\rm and} \quad
\mathcal{L}_2 \,\,= \,\, \begin{small}
\begin{pmatrix}  x & y & 0 & w \\
 y & z & -w & 0 \\
 0 & -w & x & y \\
 w & 0 & y & z \end{pmatrix} \end{small} \quad\,
 {\rm with}\,\,\,\, U = {\bf 1}_4.
  \vspace{-0.2cm}
$$
They satisfy (a),(b),(c) in Theorem \ref{thm:characterization}.
Note that $\mathcal{L}_2$ appears in \cite[eqn (39)]{Jen}. It illustrates
case (ii) in \cite[Theorem 1]{Jen}. The Jordan algebra property fails
when $-w$ is replaced by $w$ in $\mathcal{L}_2$. 
The reciprocal variety is a threefold of degree five.
\end{example}\vspace{-0.1cm}

The proof of Theorem~\ref{thm:characterization} is 
relatively short. It will be presented in
 Section~\ref{sec2}, along with all relevant definitions.
We also review classification results for Jordan algebras \cite{Jordan_1934}.
This lays the foundation for studying
the {\em Jordan locus} ${\rm Jo}(m,\SSS^n)$. 
This is the subvariety of
the Grassmannian ${\rm Gr}(m,\SSS^n)$ whose
points are the subspaces $\mathcal{L}$ 
that satisfy the closed conditions (a), (b) and (c)  above.
To be precise, let $X_1,\ldots,X_m$ represent a basis of $\Ls$ and write $\adj(U)$ for the adjoint of $U$.
From condition (b) it can be seen that $\Jo(m,\SSS^n)$ is the subvariety set-theoretically defined (in dual Stiefel coordinates) by the linear dependence~of \vspace{-0.24cm}
\[
X_1,X_2,\ldots,X_m, \,\,X_i\adj(U)X_j+X_j\adj(U)X_i\vspace{-0.1cm}
\] 
for all matrices $U\in\Ls$ and all  indices $1\leq i\leq j\leq m$.

The main goal of this article is the classification of
low-dimensional Jordan subalgebras of $\SSS^n$
and the description of the corresponding Jordan loci for small~$n$.
In Section \ref{sec3} we determine the loci of Jordan pencils $(m=2)$ and
we study their equations.
Section \ref{sec4} establishes a
lower bound on the codimension of any proper subalgebra of $(\SSS^n ,\bullet_U)$,
and it identifies the Jordan copencils $(m=4,n=3)$. In Section~\ref{sec5} we classify Jordan subalgebras of dimension $m=3$.
This is applied in Section~\ref{sec6} to study the Jordan loci ${\rm Jo}(3,\SSS^n)$.
We emphasize the explicit computation of their irreducible components and 
defining polynomials. One finding of independent interest is the
Chow matrix in Theorem \ref{thm:chow}
which represents the Chow form of the
determinantal variety  $\{ X \in \SSS^n \mid  \rk X  \leq n-2 \}$.

\section{Jordan algebras}
\label{sec2}

Fix an invertible matrix $U \in \SSS^n$. We define an algebra structure on $\SSS^n$ by setting\vspace{-0.25cm}
\begin{equation} \label{eq:bulletdef} X\bullet_UY \,\,\,:= \,\,\, \frac{1}{2}(XU^{-1}Y\,+\,YU^{-1}X)
\qquad \hbox{for all $X,Y \in \SSS^n$}.\vspace{-0.25cm}
\end{equation}
When $U$ is clear from context, we drop the subscript.
The product $\bullet$ is bilinear and commutative, with unit $U$.
Associativity fails, but, setting $\,X^{\bullet 2} \,:= \, X \bullet X$,
the following weaker version holds.
This is the axiom for a (unital) {\em Jordan algebra}:\vspace{-0.15cm}
\begin{equation} 
\label{eq:jordanaxiom}
X^{\bullet2} \bullet (X \bullet Y) \quad = \quad
X \bullet (X^{\bullet2} \bullet Y) .\vspace{-0.15cm}
\end{equation}
The role of the unit can be played by any matrix $U \in \SSS^n_{\rm inv}$.
 The resulting Jordan algebra  structure 
  is denoted by $(\SSS^n, \,\bullet_U)$. 
  Especially important for applications, notably in  statistics, optimization
  and physics, is the case  when $U$ is real and positive definite.
 Such  algebras  over $\RR$ are known as
{\em Euclidean Jordan algebras}.

\begin{remark} \label{rmk:Knotclose} 
Over the field $\RR$, the isomorphism
type of the Jordan algebra $(\SSS^n,\,\bullet_U)$ depends on the choice of $U$. To see this, 
let  $n=2$ and write $E_{ij}$ for the matrix units.
  If $U = {\bf 1}_2 = E_{11} + E_{22}$ then $(\SSS^2,\,\bullet_U)$
 has no nilpotent elements.  However, if $\,U' = E_{12} + E_{21}\, $ then the matrices
 $E_{11}$ and $E_{22}$ are nilpotent in   $(\SSS^2,\,\bullet_{U'})$.
\end{remark}

We return to this issue later: over $\CC$, the isomorphism type of $(\SSS^n,\,\bullet_U)$ is in fact independent of $U$.
First, let us prove the result in the 
introduction.

\begin{proof}[Proof of Theorem \ref{thm:characterization}]
Clearly, (c) implies (a). 
Suppose that (a) holds, so $\Linv$ is a linear space.
We shall prove that (b) holds.
Consider invertible matrices   $U,X\in\Ls$. For small $t \in \RR$, 
the inverse of $X-tU$  is given by the Neumann series\vspace{-0.25cm}
$$ 
(X-tU)^{-1}\,=\,X^{-1}({\bf 1}_n -tUX^{-1})^{-1}\,=\,
X^{-1}\sum_{k=0}^\infty t^k(UX^{-1})^k \,\,\,\,\in\, \,\, \mathcal{L}^{-1}.\vspace{-0.25cm}
$$
 Since $\Linv$ is linear and closed, the following matrix is in $\mathcal{L}^{-1}$ as well:\vspace{-0.25cm}
 $$
 X^{-1}UX^{-1} \,\, = \,\,\,
\lim_{t\to 0} \, \frac{1}{t}\bigl((X-tU)^{-1}-X^{-1}\bigr) \,\,\,\,\in\, \,\, \mathcal{L}^{-1}.\vspace{-0.25cm}
$$
Its inverse $X\bullet_{U}X$ is in $\Ls$. Since invertible $X$ are dense in 
$\Ls$, this conclusion holds for all $X\in\Ls$. Condition (b) states that $X,Y \in \mathcal{L}$
implies $X \bullet_U Y \in \mathcal{L}$. Let $X,Y\in\Ls$. Then the Jordan algebra 
squares $ (X+Y)^{\bullet2}, X^{\bullet2},  Y^{\bullet2}$ are in $\mathcal{L}$,
and we conclude that $ \,X \bullet_U Y  = \frac{1}{2} ((X+Y)^{\bullet2} - X^{\bullet2} -  Y^{\bullet2} )\,$
is also in $\mathcal{L}$. So (b) holds.

Assume (b) and consider any invertible $X\in\Ls U^{-1}$.
Then $XU\bullet_{U}XU = X^2 U$ is in $\Ls$.
Next, we see that $X^3 U =X^2U\bullet_{U}XU\in \Ls$.
By iterating this, every power $X^k$ lies in $\Ls U^{-1}$.  Since
the characteristic polynomial of $X$ has nonzero constant term, 
the Cayley-Hamilton Theorem implies $X^{-1} \in\Ls U^{-1}$.
Since invertible matrices are dense in $\Ls U^{-1}$, we get
$\Ls U^{-1} = (\Ls U^{-1})^{-1} = U\Linv$, i.e.~(c) holds.
\end{proof}

We next review material from the theory of abstract Jordan algebras.
Already in 1934, Jordan, von Neumann and Wigner~\cite{Jordan_1934}
 established a structure theorem for Euclidean Jordan algebras.
We state a variant that holds over $\CC$.
An {\em ideal} $\mathcal{I}$ in a Jordan algebra $\mathcal{A}$ is a subspace such that 
$X\bullet Y\in\mathcal{I}$ for all $X \in\mathcal{I}, \,Y \in\mathcal{A}$. The {\em radical} of $\mathcal{A}$ is
the ideal $\rad \mathcal{A}:=\{X\in \mathcal{A}\mid X\bullet Y \textit{ is nilpotent for  all $Y \in \mathcal{A}$}\}$.

\begin{theorem}\label{thm: strucure Jordan Algebras}
For a finite-dimensional complex Jordan algebra $\mathcal{A}$,
 the quotient  $\mathcal{A}/\rad \mathcal{A}$ is a direct sum of
simple Jordan algebras. Every simple Jordan algebra of dimension $\leq 8$ is  isomorphic to either $(\CC, \, \cdot \,)$, or
to $(\SSS^3, \bullet_{{\bf 1}_3})$, or to a {\em spin factor} $\CC \times \CC^d$,
 $d\geq 2$, with product $(\lambda,a)\bullet (\mu,b)=(\lambda\mu+a^\top b,\lambda b+\mu a)$
 and unit $(1,0)$.
\end{theorem}

\begin{proof}
We find this by applying known structure theorems to semi-simple Jordan algebras of dimension $\leq 8$.  See
 \cite[Ch.~X, \S 3.3]{BraunKoecher1966} or \cite[Part~I, \S 2.13]{McCrimmon2004}. \end{proof}

Theorem~\ref{thm: strucure Jordan Algebras} will be used to classify low-dimensional Jordan algebras.
First, however, we highlight the difference between $\RR$ and $\CC$
concerning Remark~\ref{rmk:Knotclose}.

\begin{lemm}\label{lemma: squareroot}
Square roots exist in any  complex Jordan subalgebra $\mathcal{L}\subset\SSS^n$ with unit $U$.
To be precise, for every matrix $X\in\mathcal{L}_{\rm inv}$ there exists $Y\in\mathcal{L}$ with $Y^{\bullet2}=X$.
\end{lemm}

\begin{proof}
The subalgebra $\mathcal{B}=\spa\{X^{\bullet\ell}\mid \ell\geq0\}$ of $\mathcal{L}$ is associative as $Y\bullet Z= YU^{-1}Z$ $=ZU^{-1}Y$ for all $Y,Z\in\mathcal B$. By \cite[Part~I, \S 2.13]{McCrimmon2004}, the only associative simple Jordan algebra in $\CC$. So by Theorem~\ref{thm: strucure Jordan Algebras}, the quotient $\mathcal{B}/\rad \mathcal{B}$ is a direct sum of copies of $\CC$. In particular, every element of $\mathcal{B}/\rad \mathcal{B}$ is a square. So there exist  $Z \in \mathcal{B}$ and $W\!\in \rad \mathcal{B}$ with $Z^{\bullet2}+W=X$. We have $X^{-1}\in U^{-1}\mathcal{B}U^{-1}$ and so, since $W\in\rad\mathcal{B}$, the matrix $WX^{-1}$ is nilpotent. Hence the matrix $X-W = Z^{\bullet2}$ is invertible and using Theorem~\ref{thm:characterization} we can take $Z^{\bullet-2}:= U\bullet_{Z^{\bullet 2}}U\in\mathcal{B}$ which satisfies $Z^{\bullet2}\bullet Z^{\bullet -2}=U$.
 Since $W\in\rad \mathcal{B}$, the element $Z^{\bullet-2}\bullet W$ is nilpotent, say $(Z^{\bullet-2}\bullet W)^{\bullet p}=0$. Using the Taylor series
 $\sqrt{1+\lambda}= \sum_{\ell=0}^\infty a_\ell \lambda^\ell $, we~see that\vspace{-0.35cm}
$$ 
\begin{matrix}
X \,\,\,=\,\,\,Z^{\bullet2}\bullet\left(U+Z^{\bullet-2}\bullet W\right)\,\,\,=\,\,\,
Z^{\bullet2}\bullet \left(\,\sum_{\ell=0}^{p-1} a_\ell \left(Z^{\bullet-2}\bullet W\right)^{\bullet\ell}\right)^{\bullet2} . \end{matrix}\vspace{-0.25cm}
$$
Hence $X$ has a square root $Y$ in the subalgebra $\mathcal{B}$ of $\mathcal{L}$.
\end{proof}

Having access to a square root within the algebra, we can find isomorphisms 
between $(\Ls,\bullet_U)$ and $(\Ls,\bullet_V)$ for different choices of  matrices $U,V\in\Ls_{\rm inv}$.

\begin{proposition} \label{prop:compU} Consider a complex Jordan algebra $(\Ls,\bullet_U)$
as in Theorem~\ref{thm:characterization}. There is an isomorphism
    $(\Ls,\bullet_U)\cong(\Ls,\bullet_V)$ of Jordan algebras for all $V \in \Ls_{\rm inv}$.
\end{proposition}

\begin{proof}
By Lemma \ref{lemma: squareroot}, there exists $W\in\Ls_{\rm inv}$ with $WU^{-1}W=V$.
Consider the~map\vspace{-0.15cm}
\[
\phi : \Ls \rightarrow \Ls \,, \,\,\,
X\,\,\mapsto \,\,VW^{-1}XW^{-1}V\,\,=\,\,2V\bullet_W(X\bullet_WV)-(V\bullet_WV)\bullet_W X. \vspace{-0.15cm}
\]
This linear automorphism of $\Ls$ 
is  invertible since $VW^{-1}$ is invertible.  We have\vspace{-0.15cm}
\[
\phi(XU^{-1}X)
\,=\,\phi(X W^{-1} V W^{-1} X)
\,=\,VW^{-1}XW^{-1}VW^{-1}XW^{-1}V
\,=\,\phi(X)V^{-1}\phi(X). \vspace{-0.15cm}
\]
This shows that $\phi$ is a Jordan algebra isomorphism from $(\Ls,\bullet_U)$ to $(\Ls,\bullet_V)$.
\end{proof}

We now come to the classification of low-dimensional Jordan algebras. 

\begin{lemm} \label{lem:k+1}
If $k=\dim(\rad\mathcal{A})$, then $X^{\bullet k+1}=0$ for all $X\in\rad\mathcal{A}$.
\end{lemm}

\begin{proof}
Let $X\in\rad\mathcal{A}$. As $X,X^{\bullet2},\ldots,X^{\bullet k+1}\in\rad\mathcal{A}$ are linearly dependent,  the minimum polynomial of $X$ has degree\,$\leq k+1$. As $X$ is nilpotent, $X^{\bullet k+1}=0$.
\end{proof}

\begin{proposition}\label{thm:2d_classification}
Every Jordan algebra $\mathcal A$ of dimension two over $\CC$ is isomorphic to 
the Jordan algebra $\CC\{U,X\}$ with unit $U$, where the  product is given by
$
\begin{array}{llllll}
1\colon\!\!\!\!&X^{\bullet2} =X,&\mbox{\!\!when }\mathcal{A}\cong\CC\times\CC; \mbox{ or}\\
2\colon\!\!\!\!&X^{\bullet2} =0,&\mbox{\!\!when }\mathcal{A}/\rad\mathcal{A}\cong\CC.
\end{array}
$
\end{proposition}

\begin{proof}
By Theorem~\ref{thm: strucure Jordan Algebras}, either
$\mathcal{A}\cong\CC\times\CC$ or $\mathcal{A}/\rad\mathcal{A}\cong\CC$.
We seek $X$ such that $\{U,X\}$ is a basis of $\mathcal{A}$. In the first case, take
the idempotent  $X=(1,0)$  in $\CC\times\CC$. In the second case, we choose $X$ to span $\rad\mathcal{A}$.
 Then $X^{\bullet 2}=0$ by Lemma \ref{lem:k+1}.
\end{proof}

\begin{theorem}\label{thm:3d_classification}
Every Jordan algebra $\mathcal A$ of dimension three over 
$\CC$ is isomorphic to the Jordan algebra $\CC\{U,X,Y\}$ 
with unit $U$, where the product is given~by
$
\begin{array}{llllll}
1\!\!\!\!&{\rm (a)}\colon\!\!\!\!&X^{\bullet2} =X,&\!\!\!\! Y^{\bullet2}=Y& \mbox{\!\!\!\!and } X\bullet Y=0,&\mbox{\!\!when }\mathcal{A}\cong\CC\times\CC\times\CC; \\
&{\rm (b)}\colon\!\!\!\!&X^{\bullet2} =U,&\!\!\!\! Y^{\bullet2}=U& \mbox{\!\!\!\!and } X\bullet Y=0,&\mbox{\!\!when }\mathcal{A}\cong\CC\times\CC^2; \\
2\!\!\!\!&{\rm (a)}\colon\!\!\!\!&X^{\bullet2} =X,&\!\!\!\! Y^{\bullet2}=0& \mbox{\!\!\!\!and } X\bullet Y=0\mbox{ or}&\\
&{\rm (b)}\colon\!\!\!\!&X^{\bullet2} =X,&\!\!\!\! Y^{\bullet2}=0& \mbox{\!\!\!\!and } X\bullet Y=Y/2,&\mbox{\!\!when }\mathcal{A}/\rad\mathcal{A}\cong\CC\times\CC; \mbox{ or}\\
3\!\!\!\!&{\rm (a)}\colon\!\!\!\!&X^{\bullet2} =Y,&\!\!\!\! Y^{\bullet2}=0& \mbox{\!\!\!\!and } X\bullet Y=0\mbox{ or} \\
&{\rm (b)}\colon\!\!\!\!&X^{\bullet2} =0,&\!\!\!\! Y^{\bullet2}=0& \mbox{\!\!\!\!and } X\bullet Y=0,&\mbox{\!\!when }\mathcal{A}/\rad\mathcal{A}\cong\CC.
\end{array}
$
\end{theorem}

\begin{proof}
By Theorem~\ref{thm: strucure Jordan Algebras},  there are four possible cases to consider:
 \vspace{-0.25cm}
\[
\mathcal{A}\,\cong\,\CC\times\CC\times\CC,~ \mathcal{A}\,\cong\,\CC\times\CC^2,
~\mathcal{A}/\rad\mathcal{A}\,\cong\,\CC\times\CC
\,\, \mbox{ or } \,\,\mathcal{A}/\rad\mathcal{A}\,\cong\,\CC.\vspace{-0.25cm}
\]
We seek $X,Y$  such that $\{U,X,Y\}$ is a basis of $\mathcal{A}$.
 In the first case, the choices $X=(1,0,0)$ and $Y=(0,1,0)$ in $\CC\times\CC\times\CC$
 have the desired properties. In the second case, we can choose $X,Y$ to be elements of order two with product zero.
Two such elements in $\CC\times\CC^2$ are $(0,e_1),(0,e_2)$. So, case 1(a) or case 1(b) holds. 

Now suppose $\mathcal{A}/\rad\mathcal{A}\cong\CC\times\CC$.
Let $X \in \mathcal{A}$ be a preimage of $(1,0)$, so
$X^{\bullet2}\equiv X\!\!\mod\rad\mathcal{A}$. Choose $Y$ to span $\rad\mathcal{A}$. Then $Y^{\bullet2}=0$ and $X\bullet Y\in\rad\mathcal{A}$. Here we use that $\rad\mathcal{A}$ is an ideal. So $X^{\bullet2}=X+\lambda Y$ and $X\bullet Y=\mu Y$ for some $\lambda,\mu\in\CC$. 

We determine all $\lambda,\mu \in \CC$ such that the commutative bilinear form $\bullet$ satisfies \vspace{-0.20cm}
$$
Z^{\bullet2}\bullet(Z\bullet W)=Z\bullet(Z^{\bullet2}\bullet W)\vspace{-0.10cm}
$$
 with $Z=c_1U+c_2X+c_3Y$ and $W=d_1U+d_2X+d_3Y$ for all $c_1,c_2,c_3,d_1,d_2,d_3\in\CC$.
 Using computer algebra,
  we find that either $\mu=0$, $\mu=1$ or $(\lambda,\mu)=(0,1/2)$. The first two cases are isomorphic via the base change that replaces $X$ by $U-X$ or $X\pm\lambda Y$. From this we can conclude that case 2(a) holds or case 2(b) holds.
%
%
%
%

If $\mathcal{A}/\rad\mathcal{A}\cong\CC$, we choose $X,Y$ to span $\rad\mathcal{A}$. In this case,  either the square of every element in $\rad\mathcal{A}$ is $0$, or we can choose $X,Y$ such that $X^{\bullet2}=Y$, since $X$ is nilpotent. In the first case, we find that $X^{\bullet2}=Y^{\bullet2}=X\bullet Y=0$. In the second case, we find that $Y^{\bullet2}=X\bullet Y=0$ as $X^{\bullet 3}=0$. So case 3(a) or case 3(b) holds.
\end{proof}

\section{Jordan Pencils}
\label{sec3}

We now embark on the study of the Jordan locus ${\rm Jo}(m,\SSS^n)$
in ${\rm Gr}(m,\SSS^n)$. It is convenient to work in
the sub-Grassmannian
${\rm Gr}_{\bf 1}(m,\SSS^n)$ of subspaces $\mathcal{L}$
that contain the identity matrix ${\bf 1}_n$. The 
{\em restricted Jordan locus} is the intersection
\begin{equation}
\label{eq:restrictedJordan} {\rm Jo}_{\bf 1}(m,\SSS^n) \,\,\, =\,\,\,
{\rm Gr}_{\bf 1}(m,\SSS^n) \, \cap \, {\rm Jo}(m,\SSS^n). 
\end{equation}

We next examine the variety (\ref{eq:restrictedJordan}) for $m=2$.
Pencils in ${\rm Gr}_{\bf 1}(2,\SSS^n)$ have the form
$\mathcal{L} = \CC \{ {\bf 1}_n, X \}$, where ${\rm trace}(X) = 0$.
This identifies  ${\rm Gr}_{\bf 1}(2,\SSS^n)$ with the projective space
$ \PP^{\binom{n+1}{2}-2}$ of traceless symmetric matrices.
With these conventions, 
$\mathcal{L}$ is a Jordan algebra if and only if
the minimal polynomial of $X$ has degree two.

\begin{proposition}
The restricted Jordan locus ${\rm Jo}_{\bf 1}(2,\SSS^n)$ is a variety with
 $\lfloor n/2 \rfloor$ irreducible components $V_i$ for $1 \leq i \leq n/2$.
The component $V_i$ has codimension $\binom{i+1}{2} + \binom{n-i+1}{2} - 2$.
It  is parametrized by diagonalizable matrices
that have two distinct eigenvalues of multiplicities $i$ and $n-i$.
This is case 1 in Proposition \ref{thm:2d_classification}.
\end{proposition}

\begin{proof}
Following \cite{FMS}, regular pencils are classified by their Segre symbols $\sigma$.
By \cite[Theorem 3.2]{FMS}, a pencil $\mathcal{L}$ has 
${\rm deg}(\mathcal{L}^{-1}) = 1$ if and only if
its Segre symbol is $\sigma = [(1,\ldots, 1), (1 ,\ldots, 1)]\,$ or
$\,\sigma = [(2,\ldots, 2 ,1 ,\ldots, 1)]$. The varieties $V_i$ are the
closures of the former strata.
These contain the latter in their closure, by \cite[Theorem 5.1]{FMS}.
The formula for ${\rm codim}(V_i)$ appears in  \cite[Proposition~5.4]{FMS}.
\end{proof}

\begin{remark}
Nondiagonalizable pencils in ${\rm Jo}_{\bf 1}(2,\SSS^n)$
have only one eigenvalue and Jordan blocks of sizes one or two.
This is case 2 in Proposition \ref{thm:2d_classification}.
\end{remark}

Equations defining  $V_i$ are computed as follows.
Fix a symmetric $n \times n$ matrix of unknowns, $X  = (x_{ij}) $.
Consider the entries of the product $(X-a {\bf 1}_n)(X-b {\bf 1}_n)$,
the $(i+1)$-minors of $X-a {\bf 1}_n$, and the
$(n-i+1)$-minors of $X-b {\bf 1}_n$. By eliminating $a$ and $b$
from these equations, we obtain homogeneous polynomials
in the entries of $X$ that cut out $V_i$ set-theoretically. We
discuss these varieties for $n \leq 5$.
 All pencils of binary quadrics $(n=2)$ are Jordan: ${\rm Jo}_{\bf 1}(2,\SSS^2) = V_1 ={\rm Gr}_{\bf 1}(2,\SSS^2)$.
 
\bigskip

We thus start with Jordan pencils of plane conics ($n=3$).

\begin{example}[$m=2,n=3$]
The variety ${\rm Jo}_{\bf 1}(2,\SSS^3) = V_1$ is irreducible of codimension $2$
and degree $4$. Its ideal is generated by the following seven cubics: \vspace{-0.17cm}
$$ \begin{small} \begin{matrix}
     x_{11} x_{13} x_{22}-x_{11} x_{13} x_{33}-x_{12}^2 x_{13}+x_{12} x_{22} x_{23}
    -x_{12} x_{23} x_{33}+x_{13}^3-x_{13} x_{22}^2+x_{13} x_{22} x_{33}, \\
     x_{11} x_{12} x_{22}-x_{11} x_{12} x_{33}-x_{12}^3+x_{12} x_{13}^2
    -x_{12} x_{22} x_{33}+x_{12} x_{33}^2+x_{13} x_{22} x_{23}-x_{13} x_{23} x_{33}, \\
     x_{11}^2 x_{23}-x_{11} x_{12} x_{13}-x_{11} x_{22} x_{23}-x_{11} x_{23} x_{33}
    +x_{12} x_{13} x_{22}+x_{13}^2 x_{23}+x_{22} x_{23} x_{33}-x_{23}^3, \\
    x_{12}^2 x_{23}-x_{12} x_{13} x_{22}+x_{12} x_{13} x_{33}-x_{13}^2 x_{23}, \quad
x_{11} x_{13} x_{23}-x_{12} x_{13}^2+x_{12} x_{23}^2-x_{13} x_{22} x_{23}, \\ \!\!\!\!
x_{11} x_{12} x_{23}-x_{12}^2 x_{13}-x_{12} x_{23} x_{33}+x_{13} x_{23}^2\,, \quad
 x_{11}^2 x_{22}-x_{11}^2 x_{33}-x_{11} x_{12}^2  +x_{11} x_{13}^2 \quad \\ \qquad -x_{11} x_{22}^2
+x_{11} x_{33}^2+x_{12}^2 x_{22}-x_{13}^2 x_{33}+x_{22}^2 x_{33}
-x_{22} x_{23}^2-x_{22} x_{33}^2+x_{23}^2 x_{33}.
\end{matrix}
\end{small} \vspace{-0.13cm}
$$
These are the expressions in the sum of squares representation of the
discriminant of the characteristic polynomial of $X$, seen in \cite[page 97]{CBMS}.
Indeed, $V_1$ is also the Zariski closure of all \underbar{real}
 matrices in $\SSS^3$ with a double eigenvalue.
 \end{example}

\begin{example}[$m=2,n=4$]
We compute  ${\rm Jo}_{\bf 1}(2,\SSS^4) $ in {\tt Macaulay2} as follows: \vspace{-0.13cm}
\begin{small}
\begin{verbatim}
  R = QQ[a,b,x11,x12,x13,x14,x22,x23,x24,x33,x34,x44];
  X = matrix {{x11,x12,x13,x14},{x12,x22,x23,x24},
              {x13,x23,x33,x34},{x14,x24,x34,x44}};
  I = eliminate({a,b},minors(1,(X-a)*(X-b)))
\end{verbatim}
\end{small} \vspace{-0.13cm}
The ideal ${\tt I}$ is generated by $30$ cubics. It is the intersection
of two prime ideals, corresponding to 
${\rm Jo}_{\bf 1}(2,\SSS^4) = V_1 \cup V_2$.
The component $V_1$ has codimension $5$ and degree $8$. Its
prime ideal is generated by $10$ quadrics, including
$\,x_{13} x_{24}-x_{12} x_{34},\, x_{14} x_{23}-x_{12} x_{34}$ and 
$x_{13} x_{23}-x_{14} x_{24}-x_{12} x_{33}+x_{12} x_{44}$.
The component $V_2$ has codimension $4$ and degree  $6$. Its prime ideal
is generated by $9$ quadrics.
\end{example}

\begin{example}[$m=2,n=5$] The ideal of 
${\rm Jo}_{\bf 1}(2,\SSS^5) = V_1 \cup V_2$ is generated by $81$ cubics.
The variety $V_1$ has codimension $9$ and degree $16$.
Its ideal is generated by $35$ quadrics, including ten 
$2 \times 2$-minors, 
like $x_{12} x_{34}-x_{13} x_{24}$.
 The variety $V_2$ has codimension $7$ and degree $40$.
Its  ideal is generated by $95$ cubics, e.g.~$ \begin{small}
x_{12} x_{13} x_{35}
- x_{13} x_{15} x_{23}
+ x_{22} x_{23} x_{35}
- x_{23}^2 x_{25}
- x_{23} x_{34} x_{45}
- x_{23} x_{35} x_{55}
+ x_{24} x_{34} x_{35}
+ x_{25} x_{35}^2.
\end{small}
$
\end{example}

We now observe that ${\rm Jo}(m,\SSS^n)$ is the orbit of (\ref{eq:restrictedJordan}) under the
congruence action by ${\rm GL}(n)$. This implies the following result in the
Grassmannian ${\rm Gr}(2,\SSS^n)$.

\begin{corollary} \label{cor:n2floor}
The Jordan locus ${\rm Jo}(2,\SSS^n)$ has $\lfloor n/2\rfloor$ irreducible components.
Using the notation from  \cite[Section 5]{FMS}, these components are 
the Grassmann strata $\overline{{\rm Gr}_\sigma}$  that are
associated with the Segre symbols
$\sigma = [(1,\ldots, 1), (1 ,\ldots, 1)]$.
\end{corollary}

\begin{remark}
Jordan algebras can be viewed as a nonabelian generalization of {\em partition matroids}.
These are the matroids that are direct sums of uniform matroids. Indeed,
the diagonalizable types in Corollary \ref{cor:n2floor} correspond to 
the rank-$2$ partition matroids, and those in 
Theorem \ref{thm:classification m=3,n=4}
to the rank-$3$ partition matroids.
\end{remark}

\section{Jordan Copencils}
\label{sec4}

This section concerns
Jordan algebras of low codimension.
These are quite rare:

\begin{theorem} \label{Thm: codimension}
Let $\Ls\subset\SSS^n$ be a proper Jordan subalgebra. Then $\codim(\Ls)\geq n-1$.
\end{theorem}

To prove this theorem, we need the following result
and some  lemmas.\vspace{-0.08cm}

\begin{theorem}[Peirce decomposition, {\cite[Part II, Chapter 8]{McCrimmon2004}}]
\label{thm:peirce}
Let $X_1,\dots X_d$ be orthogonal idempotents in $\mathcal{L}$ with $U=X_1+\dots+X_d$. Then $\Ls=\bigoplus_{1\leq i\leq j\leq d} \mathcal{X}_{i,j}$ where 
 $\,\mathcal X_{i,i}=\{\,Y\in \Ls\,\mid \,Y= X_i\bullet Y\}\,$ and
$\,\mathcal X_{i,j}=\{\,Y\in \Ls\,\mid \, Y = 2 X_i\bullet Y= 2 X_j\bullet Y\}\,$ for $\,i<j$.
\end{theorem}

Here, being  {\em orthogonal} means that $X_i \bullet X_j =0$ whenever $i \not= j$. An idempotent $X$ is called 
{\em primitive} if there are no nonzero orthogonal idempotents $X_1$ and $ X_2$ with $X=X_1+X_2$. 
We denote by $J_n$ the $n\times n$ matrix with ones on its anti-diagonal. 

\begin{lemm}[{\cite[Lemma 1]{bukovsek-omladic}}]\label{lemma:conjugation0}
If  $X,Y\in\SSS^n$ are similar, then they are orthogonally congruent.
\end{lemm}

\begin{lemm}\label{lemm:antitriagle}
If the unit $U$ of a Jordan subalgebra $\Ls \subset \SSS^n$ is primitive, then $\Ls$ is congruent to $\,\CC J_n\oplus \Ls'$ where $\Ls'$ consists of upper anti-triangular matrices.
\end{lemm}

\begin{proof} After congruence, we may assume $U = {\bf 1}_n$.
Let $X\in\Ls$.  By considering the Jordan canonical form of $X$, adding multiples of ${\bf 1}_n$ to $X$, and taking powers, we see that ${\bf 1}_n$ cannot be primitive if $X$ has distinct eigenvalues. So $\Ls=\CC {\bf 1}_n\oplus \Ls'$ where $\Ls'\subset\Ls$ is the subset of nilpotent matrices. By Lemma~\ref{lemma:conjugation0}, we see that $XY+YX\in\Ls'$ for  $X,Y\in\Ls'$. So $\Ls'$ is anti-triangularizable by \cite[Theorem 9]{bukovsek-omladic}. The transformation given there takes the unit matrix ${\bf 1}_n$ to the matrix $J_n$.
\end{proof}

Two spaces $\Ls,\Ls'\subset\SSS^n$ are congruent if $\Ls'=P\Ls P^\top$ for some $P\in\GL(n)$. 
We say that $\Ls,\Ls'$ are {\em orthogonally congruent} if $P$ lies in
the orthogonal group $O(n)$.\vspace{-0.08cm}

\begin{lemm}\label{lemma:conjugation}
If  $\Ls,\Ls'\subset\SSS^n$ are similar, then they are orthogonally congruent.
\end{lemm}

\begin{proof}
We have $\Ls P=P\Ls'$ for some  $P\in\GL(n)$. 
Consider  $A\in\Ls$ and $B\in\Ls'$ with $AP=PB$. By the proof of \cite[Lemma 1]{bukovsek-omladic}, we get $AU=UB$ for $U=(PP^\top)^{-1/2}P$.
The matrix $U$ is orthogonal, and it maps $\mathcal{L}$ into $\mathcal{L'}$
under congruence.
\end{proof}

\begin{lemm}\label{lemma:idempotent}
Let $X\in\SSS^n$ be an idempotent matrix of rank $r$. Then $X$ is orthogonally congruent to $\Diag({\bf 1}_r,{\bf 0}_{n-r})$. If $X=\Diag({\bf 1}_r,{\bf 0}_{n-r})$ and $Y\in\SSS^n$ is a matrix with $XY+YX=0$, then $Y=\Diag({\bf 0}_r,Z)$ for some $Z\in\SSS^{n-r}$.
\end{lemm}

\begin{proof} This follows from   Lemma~\ref{lemma:conjugation0}.
\end{proof}

\begin{proof}[Proof of Theorem \ref{Thm: codimension}]
Either $\Ls$ contains a primitive idempotent of rank $>1$, or
the unit $U$ of $\Ls$ is a sum of orthogonal rank-$1$ idempotents. We shall examine the associated Peirce decompositions 
(Theorem \ref{thm:peirce}). These have $d=2$ and $d=n$.

Suppose that $\Ls$ contains a primitive idempotent $X_1$ of rank $r>1$ and take $X_2=U-X_1$. Then $\Ls=\mathcal{X}_{1,1}\oplus\mathcal{X}_{1,2}\oplus\mathcal{X}_{2,2}$. By Lemmas \ref{lemm:antitriagle} and \ref{lemma:idempotent}, we assume that $X_1=\Diag(J_r,{\bf 0}_{n-r})$, $X_2=\Diag({\bf 0}_r,{\bf 1}_{n-r})$ and $\mathcal{X}_{1,1}$ consists of upper anti-triangular $r \times r$ matrices. So $\mathcal{X}_{1,1}$ has codimension $\geq r-1$ in $\SSS^r$. Elements in $\mathcal X_{1,2}$ are of the form $Z={\footnotesize
 \begin{pmatrix}
 {\bf 0}_r & V^\top\\
 V &  {\bf 0}_{n-r} 
 \end{pmatrix}
}$ with $Z^{\bullet 2}=\Diag(V^\top V,VJ_rV^\top)$. It follows that $V^\top V\in\mathcal{X}_{1,1}$.
So, the last column $v$ of $V$ satisfies $v^\top v=0$. Hence $v$ is contained in a subspace $\mathcal Y\subset\CC^{n-r}$ of codimension $c\geq (n-r)/2$. We have $\mathcal{X}_{2,2}=\Diag({\bf 0}_r,\mathcal{Z})$ where $\mathcal{Z}\subset\SSS^{n-r}$ consists of matrices that map $\mathcal Y$ into itself. Therefore $\mathcal{Z}$ has codimension $\geq c(n-r-c)$. So $\Ls$ has codimension $\geq r-1+c+c(n-r-c)\geq n-1$.

If all primitive idempotents of $\Ls$ have rank $1$, then $U=X_1+\dots+X_n$ for rank-$1$ orthogonal idempotents $X_1,\ldots,X_n$. By Lemma~\ref{lemma:idempotent}, we may assume $X_i=E_{ii}$. We 
note that the spaces of the Peirce decomposition are either $\mathcal{X}_{i,j} =\CC\{E_{i,j}+E_{j,i}\}$ or $\mathcal{X}_{i,j}=\{0\}$.  Since $\Ls$ is a proper subalgebra, at least one of these spaces is $\{0\}$, say $\mathcal X_{1,2}=\{0\}$. Since $2(E_{i,j}+E_{j,i})\bullet (E_{i,k}+E_{k,i})=E_{k,j}+E_{j,k}$, this implies either $\mathcal X_{1,i}=\{0\}$ or $\mathcal X_{2,i}=\{0\}$ for each $i \in \{3,\dots, n\}$. Therefore $\codim (\Ls)\geq n-1$.
\end{proof}

A linear subspace $\mathcal{L}\subset\SSS^n$ is said to be a 
 {\em copencil} if $\codim (\Ls) = 2$, i.e.~the orthogonal complement $ \mathcal{L}^\perp =  \{\,
X \in \SSS^n \mid {\rm trace}(XZ) = 0 \,\,\hbox{for} \,\, Z \in \mathcal{L} \} $ is a pencil.
We say that $\mathcal{L}$ is a {\em Jordan copencil} if $\mathcal{L}$ is a copencil that satisfies the equivalent
conditions (a), (b) and (c) from Theorem \ref{thm:characterization}.
Theorem \ref{Thm: codimension} implies:

\begin{corollary} There are no Jordan copencils unless $n \leq 3$.
Every copencil with $n=2$ is a Jordan copencil since it is the span of
one invertible matrix $U $ in $\SSS^2$.
\end{corollary}

It remains to study $n=3$. We represent  Jordan  copencils  by the variety
\begin{equation}
\label{eq:copencils3}
 \left\{\, (X,Y) \in \SSS^3 \times \SSS^3 \,\,\middle|\,\,
(\CC \{X,Y\} )^\perp \,\,\hbox{is a point in} \,\, \,{\rm Jo}(4,\SSS^3) \, \right\} .
\end{equation}

\vspace{5pt}

\begin{proposition}\label{thm:classification m=4,n=3}
There are two congruence orbits of Jordan copencils in $\SSS^3$:\vspace{-0.15cm}
$$ \mathcal{L}_1 \,\,=\,\, \begin{small} \begin{pmatrix}
\, x & 0 & 0 \, \\ \,0 &y&w \, \\ \, 0 &w&z\,
\end{pmatrix} \end{small} \quad \mbox{ and } \quad
\mathcal{L}_2\, \,= \,\,
\begin{small} \begin{pmatrix}
\, x&y&w \, \\ \, y&z & 0 \,\\ \, w & 0 & 0 \,
\end{pmatrix}. \end{small}\vspace{-0.15cm}
$$
The orbit of $\mathcal{L}_2$ is in the closure of
the orbit of  $\mathcal{L}_1$. These orbits have codimensions $5$ and $4$, so the Jordan 
locus ${\rm Jo}(4, \SSS^3)$ is irreducible of codimension~$4$. 
The prime ideal of (\ref{eq:copencils3}) has degree $21$ and is
 generated by $4$ cubics and $15$ quartics.
\end{proposition}

\begin{proof}
Theorem~\ref{thm: strucure Jordan Algebras} gives a finite list of cases to examine.
By multiple applications of Lemma~\ref{lemma:idempotent}, we see that $\CC\times\CC\times\CC\times\CC$ cannot be embedded into~$\SSS^3$. By Theorem~\ref{thm:classification m=3,n=n}, $\CC\times\CC^2$ cannot be embedded into $\SSS^3$.
And, a computation shows that the Jordan algebra $\CC\{X,Y\}$ with zero bilinear form $\bullet$ cannot be embedded into $\SSS^3$. Using the help of a computer as in the proof of Theorem~\ref{thm:3d_classification}, we generate a finite list of 
Jordan algebras that do not contain any of these Jordan algebras as a subalgebra. Using Lemma~\ref{lemma:idempotent}, we find that all embeddings of these Jordan algebras are 
congruent  to $\mathcal{L}_1$ or $\mathcal{L}_2$.

The last two statements are proved by computer algebra. 
Note that 
$\mathcal{L}_1^\perp $ and $\mathcal{L}_2^\perp$ are singular pencils.
 The four cubics in the ideal for
$\mathcal{L}_1$ are the four coefficients of  ${\rm det}(x X + y Y)$. The
$15$ quartics can also be derived using linear algebra. For the containment of orbits, we
check that $\mathcal{L}_2$ satisfies these $19$ equations.
\end{proof}

\section{Three-Dimensional Jordan Algebras}
\label{sec5}

The conditions in Theorem~\ref{thm:characterization} are invariant under congruence.
Ideally, one would classify Jordan subalgebras of $\SSS^n$ up to congruence and 
describe the poset of orbit closures.
We tackle this problem for $m=3$.
The following two theorems are our main results in Section \ref{sec5}.
The first gives a complete solution for $4 \times 4$ matrices.
The second gives a solution for Jordan subalgebras without a radical.

\begin{theorem}\label{thm:classification m=3,n=4}
Every Jordan net in $\SSS^4$ is congruent to one of the following nets:

\noindent$
\begin{array}{llllll}
1\!\!\!\!&{\rm (a)}\!\!\!\!\!\!\!\!\!\!&\colon\!\!\!\!&\Diag(x{\bf 1}_2,y,z);\\
&{\rm (b)}\!\!\!\!\!\!\!\!\!\!&\colon\!\!\!\!&\mbox{the Kronecker product of  ${\bf 1}_2$
with $\SSS^2$;
\qquad (see $\mathcal{L}_2$ in Example   \ref{ex:drei})}
\\
2\!\!\!\!&{\rm (a1)}\!\!\!\!\!\!\!\!\!\!&\colon\!\!\!\!&\Diag(xJ_2+yE_{11},z{\bf 1}_2);\\
&{\rm (a2)}\!\!\!\!\!\!\!\!\!\!&\colon\!\!\!\!&\Diag(xJ_3+yE_{11},z);\\
&{\rm (b)}\!\!\!\!\!\!\!\!\!\!&\colon\!\!\!\!&\Diag(xJ_2,yJ_2)+z(E_{13}+E_{31});\\
3\!\!\!\!&{\rm (a)}\!\!\!\!\!\!\!\!\!\!&\colon\!\!\!\!&x\Diag(J_3,1)+y(E_{12}+E_{21})+zE_{11};\\
&{\rm (b1)}\!\!\!\!\!\!\!\!\!\!&\colon\!\!\!\!&xJ_4+yE_{11}+zE_{22};\mbox{ or}\\
&{\rm (b2)}\!\!\!\!\!\!\!\!\!\!&\colon\!\!\!\!&xJ_4+yE_{11}+z(E_{13}+E_{31}).
\end{array}
$

\smallskip

\noindent The numberings of the congruence classes are the same as in
   Theorem \ref{thm:3d_classification}.
Their orbit closures are subvarieties in the Grassmannian ${\rm Gr}(3,\SSS^4)$ of dimension~$21$.

\begin{samepage}

The codimensions of these Jordan loci 
and their containments are as follows:

\vspace*{-0.07cm}

\begin{center}
\begin{tikzpicture}

	\node () at (-3,2.4) {$\codim$};
\node () at (-3,1.6) {$\codim$};
\node () at (-3,.8) {$\codim$};
\node () at (-3,0) {$\codim$};	
	\node () at (-2,2.4) {$9$};
	\node () at (-2,1.6) {$10$};
	\node () at (-2,.8) {$11$};
	\node () at (-2,0) {$12$};
	
    \node (1a) at (1,2.4) {1(a)};
    \node (1b) at (4,1.6) {1(b)};
    \node (2a1) at (0,1.6) {2(a1)};
    \node (2a2) at (2,1.6) {2(a2)};
    \node (2b) at (5.,.8) {2(b)};
    \node (3a) at (1.,.8) {3(a)};
    \node (3b1) at (3.,.8) {3(b1)};
    \node (3b2) at (3.,0) {3(b2)};
    
    \draw [thick] (1a) -- (2a1);
    \draw [thick] (1a) -- (2a2);
    \draw [thick] (1b) -- (2b);
    \draw [thick] (1b) -- (3b1);
    \draw [thick] (2a1) -- (3a);
    \draw [thick] (2a2) -- (3b1);
    \draw [thick] (2b) -- (3b2);
    \draw [thick] (2a2) -- (3a);
    \draw [thick] (3a) -- (3b2);
    \draw [thick] (3b1) -- (3b2);
\end{tikzpicture}
\end{center}
\end{samepage}
\end{theorem}

\begin{corollary} \label{cor:Jo34}
The Jordan locus ${\rm Jo}(3,\SSS^4)$ has two irreducible components.
\end{corollary}

Jordan algebras of types 1(a) and 1(b) are of primary interest also for $n \geq 5$.

\begin{theorem}\label{thm:classification m=3,n=n}
Every Jordan net in $\SSS^n$ of type 1(a) is congruent to a diagonal~net
$ \Diag(x{\bf 1}_{k_1},y{\bf 1}_{k_2},z{\bf 1}_{k_3})$
for $k_1\geq k_2\geq k_3\geq1$ with sum $n$.
Jordan nets of type 1(b)~exist only for even $n$.
They are congruent to the Kronecker product of ${\bf 1}_{n/2}$ with~$\SSS^2$.
\end{theorem}

We note that
Theorem \ref{thm:classification m=3,n=n}
does not account for all irreducible components of
${\rm Jo}(3,\SSS^n)$ when $n \geq 5$. At present we do not
know these decompositions.

\begin{proposition} \label{prop:tau}
The Jordan net $\Ls_* =uJ_5+x(E_{15}+E_{51})+y(E_{14}+E_{41}) $ in $\SSS^5$ has type 2(b).
It  is not in the closure of the diagonalizable locus 1(a) in ${\rm Jo}(3,\SSS^5)$.
\end{proposition}

\begin{proof}
For any linear space $\Ls \subset \SSS^n$ we consider 
 $\tau(\Ls):=\min\{\rk X\mid X\in\Ls \setminus \{0\}\}$.
 This invariant equals $1$ for all Jordan algebras $\Ls$ of type 1(a), since
 $k_3 = 1$ for $n=5$.
 For the special net above, we see that $\tau(\Ls_*) = 2$.
 It now suffices to note that $\{\Ls\in\Gr(m,\SSS^n)\mid \tau(\Ls)\leq k\}$  is a 
 closed set because it is the projection along $\PP(\SSS^n)$ of the incidence variety
$\{(\Ls,X)\in\Gr(m,\SSS^n)\times\PP(\SSS^n)\mid X\in\Ls,\rk X\leq k\}$.
\end{proof}

The situation is easier if we restrict to the case of interest in applications.
Let $\Ls$ be a real Jordan subalgebra of symmetric matrices whose unit $U$ is positive definite. Then $\Ls$ is {\em formally real}, i.e. if $X_1^{\bullet2}+\cdots+X_k^{\bullet 2}=0$, then $X_1,\ldots,X_k=0$.  Formally real Jordan algebras have zero radical \cite{Jordan_1934} and this persists if we view
them as Jordan algebras over $\CC$.
We define the {\em Euclidean Jordan locus} ${\rm EJo}(3,\SSS^n)$ to be the closure 
in ${\Gr}(3,\SSS^n) $ of the set of such Jordan algebras.
 For our study with $m=3$, we get that formally real Jordan nets 
  have type 1(a) or~1(b). 

\begin{corollary} 
The number of irreducible components of 
${\rm EJo}(3,\SSS^n)$ is the coefficient of the $t^n$ in the generating function
\[
t^3/((1-t)(1-t^2)(1-t^3)) + t^2/(1-t^2).
\]
\end{corollary}
\begin{proof}
The first summand is the generating function of triples of $a,b,c\geq0$ with $a+2b+3c=n-3$. These correspond to triples of $k_1\geq k_2\geq k_3\geq 1$ with sum $n$ via  \vspace{-0.16cm}
\[
(a,b,c)\mapsto(a+b+c+1,b+c+1,c+1). \vspace{-0.16cm}
\]
The second generating function accounts for the Jordan nets of type 1(b).
\end{proof}

In the remainder of this section we present the proofs of our two theorems.

\begin{proof}[Proof of Theorem \ref{thm:classification m=3,n=n}]
Let $\mathcal{L}$ be a Jordan net of type 1(a). Then $\mathcal{L}$ is congruent to 
$\,\Diag(x{\bf 1}_{k_1},y{\bf 1}_{k_2},z{\bf 1}_{k_3}) $.
This is seen by two applications of Lemma~\ref{lemma:idempotent}.
Suppose $\mathcal{L}$ has unit $U = {\bf 1}_n$ and type 1(b). After applying an orthogonal congruence, we may assume that $X=\Diag({\bf 1}_r,-{\bf 1}_{n-r})$ for some $1\leq r\leq n-1$. The condition $XY+YX=0$ now implies 
$ Y= \begin{small} \begin{pmatrix} 0&P\\
P^\top&0\end{pmatrix} \end{small} $
for some $P\in\CC^{r\times(n-r)}$. The condition $Y^2={\bf 1}_n$ implies  $PP^\top={\bf 1}_r$ and $P^\top P={\bf 1}_{n-r}$. This is only possible when $2r=n$. By applying the orthogonal
 congruence $\Diag(P^\top, {\bf 1}_r)$, we get the desired form.
\end{proof}

\begin{proof}[Proof of Theorem \ref{thm:classification m=3,n=4}]
We proceed case-by-case, starting from the classification of
Jordan algebras in Theorem~\ref{thm:3d_classification}. 
Let $U,X,Y$ be a basis of $\mathcal{L}$, where we
 assume $U={\bf 1}_4$. The types 1(a) and 1(b) are already covered by the previous proof.

Suppose that $\mathcal{L}$ has type 2(a). We can assume $X=\Diag({\bf 1}_{4-k},{\bf 0}_k)$ and $Y=\Diag({\bf 0}_{4-k},Z)$ with $Z^2={\bf 0}_k$ for some $k\in\{2,3\}$.
Up to orthogonal congruence, there is only one 
$Z \in \SSS^k$ with $Z^2 = {\bf 0}_k\neq Z$, 
as there is only one option for its Jordan canonical form. 
This uses Lemma~\ref{lemma:conjugation0}.
 So, we get two classes of this type.

Suppose that $\mathcal{L}$ has type 2(b). Then we can assume that $X=\Diag({\bf 1}_k,{\bf 0}_{4-k})$. The condition $XY+YX=Y$ implies 
$ Y= \begin{small} \begin{pmatrix} 0&P\\
P^\top&0\end{pmatrix} \end{small} $
for some matrix $P\in\CC^{k\times(4-k)}$. The condition $Y^2={\bf 0}_4$ implies $PP^\top={\bf 0}_k$ and $P^\top P={\bf 0}_{4-k}$. As $P$ is nonzero, this is only possible when $k=2$. So $P=ab^\top$ for some nonzero vectors $a,b\in\CC^2$ with $a^\top a=b^\top b=0$. Such matrices $Y$ form a single orbit of $\Diag({\rm O}(2),{\rm O}(2))$. So we get one class of this type.
Suppose that  $\mathcal{L}$ has type 3(a). Then $X^2=Y$ and $X^3=0$. By Lemma~\ref{lemma:conjugation0}
and appealing to Jordan canonical forms as above, there is a unique $X$ up to 
    orthogonal congruence. So we get one class of this type.

Consider $\mathcal{L}$ of type 3(b). We have $X^2=Y^2=XY+YX=0$. Since $X$, $Y$ or~$X+Y$ has rank~$2$, we may assume $\rk X=2$. One can show, using Lemma~\ref{lemma:conjugation0},
 that we may assume $\rk Y=1$.  Using Lemma~\ref{lemma:conjugation} 
 and  conjugating $\mathcal{L}$ in $\CC^{n \times n}$, we~get\vspace{-0.15cm}
$$
X=\begin{pmatrix}
0&{\bf 1}_2\\
0&0
\end{pmatrix}\,\,\,\mbox{ and } \,\,\, Y=\begin{pmatrix}
0&Z\\0&0
\end{pmatrix}\vspace{-0.15cm}
\,\,\quad \hbox{where $\,\rk Z =1$.}
$$
 By conjugating with a matrix $\Diag(P,P)$ and scaling, we find $Z=E_{22}$ or $Z=E_{12}$. 
 So we get two classes. In one of them,  $\rad\mathcal{L}$ is spanned by two rank-$1$ matrices.

To identify the poset of inclusions,
 consider the associated nets of quadrics
 $ (a,b,c,d)\Ls(a,b,c,d)^T$.
In this notation, type 1(a) is written as $x(a^2+b^2)+yc^2+zd^2$.
Replacing $(a,b,c,d)$ by $(a,b,c,c+td)$ in 1(a), 
where $t \in \RR \setminus \{0\}$, we get \vspace{-0.15cm}
$$
x(a^2+b^2)+y'c^2+z'(2cd+td^2)\vspace{-0.15cm}
$$
for $y'=y+z$ and $z'=tz$. As $t\to0$, we obtain type 2(a1).  Similarly:

$
\begin{array}{llclcl}
\bullet\mbox{ We get ~1(a)}&\!\!\!\to\!\!\!&\mbox{2(a2)}&\mbox{using }(c+t^2a,tb,c,d).\\
\bullet\mbox{ We get ~1(b)}&\!\!\!\to\!\!\!&\mbox{2(b)}&\mbox{using }(-a,b,i(a+tc),i(b+td)).
\end{array}
$

$
\begin{array}{llclcl}
\bullet\mbox{ We get ~1(b)}&\!\!\!\to\!\!\!&\mbox{3(b1)}&\mbox{using }(a,tb,tc,d).\\
\bullet\mbox{ We get ~2(a1)}&\!\!\!\to\!\!\!&\mbox{3(a)}&\mbox{using }(a,b,a+tb+t^2c,td).\\
\bullet\mbox{ We get ~2(a2)}&\!\!\!\to\!\!\!&\mbox{3(a)}&\mbox{using }(a,td,t(tc-b),a+tb).\\
\bullet\mbox{ We get ~2(a2)}&\!\!\!\to\!\!\!&\mbox{3(b1)}&\mbox{using }(a,d+tb,c,d).\\
\bullet\mbox{ We get ~2(b)}&\!\!\!\to\!\!\!&\mbox{3(b2)}&\mbox{using }(c+ta,d+tb,c,d).\\
\bullet\mbox{ We get ~3(a)}&\!\!\!\to\!\!\!&\mbox{3(b2)}&\mbox{using }(a,b,tc,i(b-td)).\\
\bullet\mbox{ We get ~3(b1)}&\!\!\!\to\!\!\!&\mbox{3(b2)}&\mbox{using }(a,a+tb,c,-c+td).
\end{array}
$

\noindent The codimensions were found by direct computation.
It remains to show that there is no inclusion of orbit closures when
there is no edge in our Hasse diagram.
The arrow 1(a) $\rightarrow$ 2(b) is missing because $\tau$ cannot go up in limits;
see Proposition \ref{prop:tau}.
  The absence of other inclusions is certified by explicit 
  polynomials that vanish on one orbit, but not on another. For example, the polynomial for 
  the orbit 1(b)    displayed in (\ref{eq:onerep}) 
  does not vanish on the  orbits 1(a), 2(a1), 2(a2) and 3(a).
  Similarly, $4 p_{012} p_{146} - 4 p_{013} p_{145} -
   p_{014} p_{056} - p_{014} p_{126} + p_{014} p_{135} - p_{014} p_{234}
       + p_{024} p_{046} - p_{034}  p_{045}    \,   $ vanishes on
        2(a1) but not on  3(b1).
  \end{proof}

\section{Nets of Quadrics}
\label{sec6}

In this section we offer a more detailed study of Jordan loci in the case $m=3$.

\begin{proposition} \label{prop:33}
The Jordan locus ${\rm Jo}(3,\SSS^3) $ consists
of nets that are generated by three rank-$1$ conics in $\PP^2$.
This  is an irreducible variety of codimension $3$ 
in the Grassmannian ${\rm Gr}(3,\SSS^3) $, and it has
degree $57$ in the Pl\"ucker embedding in $ \PP^{19}$.
Its prime ideal is generated by $62$ quadrics
in the $20$ Pl\"ucker coordinates.
\end{proposition}

\begin{proof}
Here we use the Wall's classification \cite{Wall77} of nets of conics.
For each net $\mathcal{L}$ in his list, we computed 
$\mathcal{L}^{-1}$. The surface $\mathcal{L}^{-1}$ is  linear
only for Wall's types E, G and H. According to 
\cite[Figure 5]{Wall77}, the strata G and H are contained in the closure of
stratum E, which has codimension $3$, by \cite[Table 2]{Wall77}.
 We see in \cite[Table 1]{Wall77} that E corresponds to nets
spanned by rank-$1$ conics. Hence the closure of stratum E is
an irreducible variety. It equals ${\rm Jo}(3,\SSS^3) $.
The results about the ideal and degree of this Jordan locus are  in
\cite[Proposition 3.1]{BS}. 
\end{proof}

The ideal in Proposition \ref{prop:33} simplifies greatly if
we restrict to  $ {\rm Jo}_{\bf 1}(3,\SSS^3)$. Namely, for the net
$\mathcal{L} = {\rm span} \{{\bf 1}_3,\, X,\, Y \}$, with
$X = (x_{ij})$ and $Y= (y_{ij})$, we obtain \vspace{-0.14cm} \begin{equation}
\label{eq:buchloe} \begin{small}
\begin{matrix}
\!\! \bigl\langle \,
 (x_{11} y_{12}-x_{12} y_{11})+(x_{12} y_{22}-x_{22} y_{12})+(x_{13} y_{23}-x_{23} y_{13}), \\ \quad
 (x_{11} y_{23} - x_{23}  y_{11} )+(x_{12} y_{13}-x_{13} y_{12})+2 (x_{22} y_{23}-x_{23} y_{22}), \\ \qquad
  2 (x_{11} y_{13}-x_{13} y_{11})+(x_{12} y_{23}-x_{23} y_{12})-(x_{13} y_{22}-x_{22} y_{13}) \bigr\rangle .
\end{matrix} \end{small} \vspace{-0.14cm}
\end{equation}
This prime ideal of codimension $3$ defines the Jordan locus
in  ${\rm Gr}_{\bf 1}(3, \SSS^3)$.

We now assume $n \geq 3$.
For any net $\mathcal{L} \in {\rm Gr}(3,\SSS^n)$,
the reciprocal  $\mathcal{L}^{-1}$ 
 is a surface in $\PP (\SSS^n)
= \PP^{\binom{n+1}{2}-1}$.
This is a linear plane precisely when $\mathcal{L}$ is a Jordan algebra.
We now examine the other extreme case, when
$\mathcal{L}$ is generic in~${\rm Gr}(3,\SSS^n)$.

\begin{theorem} \label{thm:chow}
 The following conditions are equivalent for a net of quadrics $\mathcal{L}$:  \vspace{-0.15cm}
\begin{enumerate}
\item Every nonzero complex matrix in $\mathcal{L}$ has rank $\geq n-1$.  \vspace{-0.15cm}
\item The linear span of the surface $\mathcal{L}^{-1}$ is the ambient projective space $\PP(\SSS^n)$.\vspace{-0.15cm}
\item $\mathcal{L}^{-1}$ is linearly isomorphic to the $(n-1)$st Veronese embedding of $\PP^2$.
 \vspace{-0.15cm}
\item The Chow matrix of $\mathcal{L}$ is invertible.
\end{enumerate}
\end{theorem}

To explain this result, we must first define the {\em Chow matrix}.
Fix a basis $X,Y,Z$  of $\mathcal{L}$, and let $x,y,z$ be unknowns.
The Chow matrix is
square of size $\binom{n+1}{2} \times \binom{n+1}{2}$, and its
entries are homogeneous polynomials of degree $n-1$ in
the entries of $X,Y,Z$. The columns are labeled by the monomials
of degree $n-1$ in $x,y,z$, and the rows are labeled by
pairs $(i,j)$ where $1 \leq i \leq j \leq n$. Consider the entry in position
$(i,j)$  of the adjoint matrix of $xX+yY+zZ$. This is a polynomial of
degree $n-1$ in $x,y,z$. Its coefficients form the
row indexed $(i,j)$ of the Chow matrix. The entry in a given column is
the coefficient of the monomial label.

\begin{example}[$n=3$]
The Chow matrix of three plane conics has format $6 \times 6$: \vspace{-0.15cm}
$$ 
\begin{tiny} 
\begin{pmatrix}
x_{22} x_{33}-x_{23}^2 &
 x_{22} y_{33}-2 x_{23} y_{23}+x_{33} y_{22} &
  x_{22} z_{33}-2 x_{23} z_{23}+x_{33} z_{22} &
      y_{22} y_{33}-y_{23}^2 &
  \cdots & \cdots  \phantom{spa} \\
x_{13} x_{23} -x_{12} x_{33} &x_{13} y_{23}-x_{12} y_{33}+x_{23} y_{13}-x_{33} y_{12} &
 x_{13} z_{23}-x_{12} z_{33}+x_{23} z_{13}-x_{33} z_{12} &
  y_{13} y_{23} -y_{12} y_{33}&
  \cdots & \cdots \phantom{spa} \\
 x_{12} x_{23}-x_{13} x_{22} & x_{12} y_{23}-x_{13} y_{22}-x_{22} y_{13}+x_{23} y_{12} &
  x_{12} z_{23}-x_{13} z_{22}-x_{22} z_{13}+x_{23} z_{12} & y_{12} y_{23}-y_{13} y_{22} &
  \cdots & \cdots \phantom{spa} \\
x_{11} x_{33}-x_{13}^2 & x_{11} y_{33}-2 x_{13} y_{13}+x_{33} y_{11} & 
x_{11} z_{33}-2 x_{13} z_{13}+x_{33} z_{11} &
 y_{11} y_{33}-y_{13}^2 &  
 \cdots & \cdots \phantom{spa} \\
x_{12} x_{13} -x_{11} x_{23} &
x_{12} y_{13} -x_{11} y_{23}+x_{13} y_{12}-x_{23} y_{11} &
x_{12} z_{13} -x_{11}  z_{23}+x_{13} z_{12}-x_{23} z_{11} &
  y_{12} y_{13} -y_{11} y_{23} &
\cdots & \cdots \phantom{spa} \\  
x_{11} x_{22}-x_{12}^2 & x_{11} y_{22}-2 x_{12} y_{12}+x_{22} y_{11} &
 x_{11} z_{22}-2 x_{12} z_{12}+x_{22} z_{11} &
  y_{11} y_{22}-y_{12}^2 &
 \cdots & \cdots \phantom{spa} \\
\end{pmatrix}.
\end{tiny} \vspace{-0.13cm}
$$
The row labels $(1,1),(1,2),(1,3),(2,2),(2,3),(3,3)$ refer
to the entries of the adjoint of $xX+yY+zZ$. The
column labels are the monomials $x^2, xy, xz, y^2, yz, z^2$.
The determinant of the Chow matrix is an irreducible homogeneous
polynomial of degree $12$ with $22659$ terms.
It vanishes if and only if
the net contains a rank-$1$ matrix.
Thus, it equals the Chow
form \cite{FKO}  of the Veronese surface in~$\PP^5$. 
\end{example}

The {\em Chow form} of a projective variety is
a hypersurface in the Grassmannian. Its points are
linear spaces of complementary dimension that
unexpectedly intersect the variety. If the linear
space is expressed as the row space of a matrix
then the entries of the matrix are the
{\em dual Stiefel coordinates} of that linear space.

\begin{lemm} \label{lem:stiefel} Consider the net $\mathcal{L} = \{xX + yY+zZ\}$ spanned by
$X,Y,Z \in \SSS^n$.
The Chow matrix is singular if and only if $\mathcal{L}\setminus\{{\bf 0}_n\}$
contains a matrix of rank $\leq n-2$.
So the determinant of the Chow matrix is the Chow form,
written in dual Stiefel coordinates, of the subvariety of $\PP(\SSS^n)$ defined by the
$(n-1) \times (n-1)$-minors.
\end{lemm}

\begin{proof}
If $X$ has rank $\leq n-2$ then the coefficient of $x^{n-1}$ is zero
in every entry of the adjoint of $xX + yY+zZ$, so the first row
of the Chow matrix is zero. Such a rank drop happens whenever
some matrix in $\mathcal{L}$
has rank $\leq n-2$ since the determinant of the Chow matrix is invariant of the
choice of basis in $\mathcal{L}$. That determinant is a nonzero
homogeneous polynomial
in the entries of $X,Y,Z$ of degree
$(n-1)\binom{n+1}{2}$. We already argued that it vanishes on the hypersurface defined
by the Chow form of the rank-$( n-2)$ variety.
That variety has degree $\binom{n+2}{3}$, so its Chow form
 in dual Stiefel coordinates is an irreducible
polynomial of degree $3 \binom{n+2}{3}$. This irreducible polynomial
divides our determinant. Since $(n-1)\binom{n+1}{2} = 3 \binom{n+2}{3}$,
the two polynomials agree up to a nonzero constant. \end{proof}

\begin{remark}
The Chow matrix is new and interesting even for $n=4$.
It gives the Chow form for
symmetric $4 \times 4$ matrices of rank~$\leq~2$.
A formula for that Chow form was the
main result in \cite[Section 3]{FKO}. Our construction is much simpler.
 \end{remark}
 
\begin{proof}[Proof of Theorem \ref{thm:chow}]
Lemma \ref{lem:stiefel} says that  conditions (1) and (4) are equivalent.
The surface $\mathcal{L}^{-1}$ is parametrized by the entries
of the adjoint of $xX+yY+zZ$.
By inverting the Chow matrix, we can express every monomial
of degree $n-1$ in $x,y,z$ as a linear combination of those entries.
Hence $\mathcal{L}^{-1}$ is linearly isomorphic to the Veronese surface,
so (4) implies  (3). Clearly (3) implies (2). Finally,
 if (4) fails then $\mathcal{L}^{-1}$ lies in a hyperplane in $\PP(\SSS^n)$, so (2) fails.
\end{proof}

\begin{corollary} \label{cor:chowrank}
Given any regular net $\mathcal{L}$ of symmetric $n \times n$ matrices,
the rank of its Chow matrix 
equals the dimension of the linear span of $\mathcal{L}^{-1}$ inside $\SSS^n$.
\end{corollary}

\begin{proof}
A linear form vanishes on $\mathcal{L}^{-1}$ if and only if it vanishes 
on the adjoint of $xX + yY+zZ$. These are the
 vectors in the left kernel of the Chow matrix.
\end{proof}

Every regular subspace $\mathcal{L}\subset\SSS^n$ generates a unique Jordan subalgebra
${\rm Jor}(\mathcal{L})$ inside $\SSS^n$.
The dimension of ${\rm Jor}(\mathcal{L})$ is bounded below by the
rank of the Chow matrix of $\mathcal{L}$.
However, this bound is not tight. To see this, consider the~net
\begin{equation}
\label{eq:netrank8}
\mathcal{L} \,\,=\,\,\, \begin{small} \begin{pmatrix} 
\, x +y \, & z & z & 0 \,\, \\
z & \!\! x-y \!\! & 0 & z \,\, \\
z & 0 & x & z \,\,\\
0 & z & z & x \,\, \\
\end{pmatrix} 
\end{small}.
\end{equation}
Here, ${\rm Jor}(\mathcal{L})= \SSS^4$ but
 the $10 \times 10$ Chow matrix has rank~$8$.  Its left kernel
 gives the  two linear forms
 $\, z_{14}-z_{23}-z_{33}+z_{44} \,$ and $\, 2 z_{12}-z_{13}-z_{24}  \,$ which
 vanish on $\mathcal{L}^{-1}$.

\smallskip

We now return to the polynomial equations that vanish on the Jordan locus
in the Grassmannian.
The following result is immediate from Corollary \ref{cor:chowrank}.

\begin{proposition} \label{prop:choweqns} Using dual Stiefel coordinates 
on ${\rm Gr}(m,\SSS^n)$, the Jordan locus  $\,{\rm Jo}(m,\SSS^n)\,$
is cut out by the $(m+1) \times (m+1)$-minors of the Chow matrix.
\end{proposition}
\begin{proof}
For any regular $\Ls\in\Gr(m,\SSS^n)$, the reciprocal variety $\Linv$ is irreducible of dimension $m$ and contained in its span, which is itself an irreducible variety. These two varieties coincide if and only if the latter has dimension~$\leq m$.
\end{proof}

\begin{remark}
These $4 \times 4$-minors do not generate
the prime ideal of ${\rm Jo}(3,\SSS^n)$. For instance, if  
$n=3$ and we fix ${\bf 1}_3 \in \mathcal{L}$ then
Proposition \ref{prop:choweqns} yields $30$ quartics in $x_{ij},y_{ij}$.
These cut out the same variety 
${\rm Jo}_{\bf 1}(3,\SSS^3)$ as the three quadrics in (\ref{eq:buchloe}).
\end{remark}

We know from Theorem \ref{thm:classification m=3,n=4}
that  ${\rm Jo}(3,\SSS^4)$ has two irreducible components,
of types 1(a) and 1(b).
They are given by the nets $\mathcal{L}_1$ and $\mathcal{L}_2$ in the next example.

\begin{example} \label{ex:drei}
We consider the congruence orbits of the following nets in $\SSS^4$:  \vspace{-0.1cm}
$$\mathcal{L}_1 \,\,= \,\,
\begin{small} \begin{pmatrix}
x & 0 & 0 & 0 \\
0 & x & 0 & 0 \\
0 & 0 & y & 0 \\
0 & 0 & 0 & z 
\end{pmatrix} \end{small}\, , \,\,\,
\mathcal{L}_2 \,\,= \,\,
\begin{small} \begin{pmatrix}
x & y & 0 & 0 \\
y & z & 0 & 0 \\
0 & 0 & x & y \\
0 & 0 & y & z 
\end{pmatrix} \end{small}
\, , \,\,\,
\mathcal{L}_3 \,\,= \,\,
\begin{small} \begin{pmatrix}
0 & 0 & x & y \\
0 & -2x & -y & z \\
x & -y & -2z & 0 \\
y & z & 0 & 0 
\end{pmatrix} \end{small}.  \vspace{-0.1cm}
$$
They  have codimensions $9$, $10$ and $9$ in ${\rm Gr}(3,\SSS^4)$.
The first two fill out ${\rm Jo}(3,\SSS^4)$.

It is interesting to  compare the non-diagonalizable nets
$\mathcal{L}_2$ and $\mathcal{L}_3$.
Both represent double conics in $\PP^2$, since
 ${\rm det}(\mathcal{L}_2) =  {\rm det}(\mathcal{L}_3) = (xz-y^2)^2$.
 But, while $\mathcal{L}_2^{-1}$ is linear,
  the Chow matrix of $\mathcal{L}_3$ is invertible, so
 $\mathcal{L}_3^{-1}$ is a Veronese surface.
 The congruence orbit of $\mathcal{L}_3$ is the
 smooth variety of codimension $9$ 
  studied~in \cite{EPS}. It is essentially the
 Hilbert scheme of
  twisted cubic curves in $\PP^3$. 
  \end{example}

We end this section by reporting a challenging computation.
It is predicated on the idea that 
Pl\"ucker coordinates are best when working with Grassmannians.

\begin{example}\label{example: quadrics in plucker}
The $21$-dimensional Grassmannian ${\rm Gr}(3,\SSS^4)$ lives in $\PP^{119}$. 
This space has
 $\binom{10}{3} = 120$ dual Pl\"ucker coordinates $p_{ijk}$,
 one for each  $3 \times 3$-minor~of  \vspace{-0.14cm}
$$ \begin{small} \begin{pmatrix}
\, u_{11} & u_{12} & u_{13} & u_{14} & u_{22} & u_{23} & u_{24} & u_{33} & u_{34} & u_{44} \, \\
\, x_{11} & x_{12} & x_{13} & x_{14} & x_{22} & x_{23} & x_{24} & x_{33} & x_{34} & x_{44} \, \\
\, y_{11} & y_{12} & y_{13} & y_{14} & y_{22} & y_{23} & y_{24} & y_{33} & y_{34} & y_{44}\  \\
\end{pmatrix}. \end{small}  \vspace{-0.14cm}
$$
We computed the quadrics that vanish on subvarieties of ${\rm Gr}(3,\SSS^4)$.
Modulo Pl\"ucker relations, there are $4950$ linearly independent quadrics.
Each quadric has the form $\sum \lambda^{ijk}_{lmn} p_{ijk} p_{lmn}$
where the $\lambda^{ijk}_{lmn}$ are unknown coefficients and
$i \leq l, j \leq m, k \leq n$.
For each of the three orbits in 
Example \ref{ex:drei},
we set up a linear system of equations in the $\lambda^{ijk}_{lmn}$ 
 whose solutions  are quadrics that vanish on that variety.
We then solved these equations using {\tt Maple}. Here are the results:

We found a basis of $85$ quadrics for the orbit of $\mathcal{L}_1$. One representative is
\vspace{-0.16cm}
$$ \begin{small}
\begin{matrix} p_{124} p_{789}-p_{125} p_{589}+p_{125} p_{679}
   -p_{126} p_{678} + p_{127} p_{569} - p_{128} p_{568} - p_{234} p_{579} \\
      + p_{234} p_{678} + p_{235}  p_{479} + p_{235} p_{568}
    - p_{236} p_{478} - p_{236} p_{567} - p_{237} p_{459} + p_{238} p_{458}. \end{matrix}
    \end{small} \vspace{-0.16cm} $$
We found a basis for $189$ quadrics for the orbit of $\mathcal{L}_2$. One representative is \vspace{-0.16cm}
\begin{equation} \label{eq:onerep} \begin{small} \begin{matrix}
 p_{012} p_{018} - p_{012} p_{026} - p_{012} p_{035} - 2 p_{012} p_{123}
     - p_{013} p_{017} + 2 p_{013} p_{025} - p_{023} p_{024}.  \end{matrix}
     \end{small} \vspace{-0.16cm} 
\end{equation}     
Each of these $189$ quadrics also vanishes on $\mathcal{L}_3$. But, there are 
$918$ quadrics vanishing on the orbit of $\mathcal{L}_3$. One representative is
$\, 8 p_{012} p_{457} \,-\, 4 p_{045} p_{057} \,+\, p_{047}^2 $.
To see that $\mathcal{L}_3$ is not contained in $\mathcal{L}_2$,
we can use the quartics in  Proposition~\ref{prop:choweqns}.
\end{example}

\begin{small}

\end{small}

\end{document}